\documentclass[11pt]{article}
\usepackage{amsfonts, amssymb, amsmath, amsthm}
\usepackage{young}
\begin{document}

\theoremstyle{plain}
\newtheorem{thm}{Theorem}[section]
\newtheorem{cor}[thm]{Corollary}
\newtheorem{con}[thm]{Conjecture}
\newtheorem{cla}[thm]{Claim}
\newtheorem{lm}[thm]{Lemma}
\newtheorem{prop}[thm]{Proposition}
\newtheorem{example}[thm]{Example}

\theoremstyle{definition}
\newtheorem{dfn}[thm]{Definition}
\newtheorem{alg}[thm]{Algorithm}
\newtheorem{prob}[thm]{Problem}
\newtheorem{rem}[thm]{Remark}

\newcommand{\de}{\textnormal{det}}
\newcommand{\fall}{\downarrow}
\newcommand{\dimension}{\textnormal{dim~}}
\newcommand{\sign}{\textnormal{sign}}
\renewcommand{\baselinestretch}{1.1}
\newcommand{\w}{\widehat}

\title{\bf Solving the Ku-Wales conjecture on the eigenvalues of the derangement graph}
\author{
Cheng Yeaw Ku
\thanks{ Department of Mathematics, National University of
Singapore, Singapore 117543. E-mail: matkcy@nus.edu.sg} \and Kok
Bin Wong \thanks{
Institute of Mathematical Sciences, University of Malaya, 50603
Kuala Lumpur, Malaysia. E-mail:
kbwong@um.edu.my.} } \maketitle

\begin{abstract}\noindent
We give a new recurrence formula for the eigenvalues of the derangement graph. Consequently, we provide a simpler proof of the Alternating Sign Property of the derangement graph. Moreover, we prove that the absolute value of the eigenvalue decreases whenever the corresponding partition decreases in the dominance order. In particular, this settles affirmatively a conjecture of Ku and Wales (J. of Combin. Theory, Series A 117 (2010) 289--312) regarding the lower and upper bound for the absolute values of these eigenvalues.
\end{abstract}

\bigskip\noindent
{\sc keywords:} Cayley graphs, symmetric group, shifted schur functions, derangement graph\newline
\noindent
{\sc \small  2010 MSC: 05C35, 05C45}

\section{Introduction}

Let $G$ be a finite group and $S$ be a subset of $G$. The Cayley graph $\Gamma(G,S)$ is the graph which has the elements of $G$ as its vertices and two vertices $u, v \in G$ are joined by an edge if and only if $uv^{-1} \in S$. We require that $S$ is a nonempty subset of $G$ satisfying the condition that $s \in S \Longrightarrow s^{-1} \in S$ and $1 \not \in S$.

The {\em derangement graph} $\Gamma_{n}$ is the Cayley graph $\Gamma(\mathcal{S}_{n}, \mathcal{D}_{n})$ where $\mathcal{S}_{n}$ is the symmetric group on $[n]=\{1, \ldots, n\}$, and $\mathcal{D}_{n}$ is the set of derangements in $\mathcal{S}_{n}$. That is, two vertices $g$, $h$ of $\Gamma_{n}$ are joined if and only if $g(i) \not = h(i)$ for all $i \in [n]$, or equivalently $gh^{-1}$ fixes no point.

Clearly, $\Gamma_{n}$ is vertex-transitive, so it is $D_{n}$-regular where $D_{n} = |\mathcal{D}_{n}|$. It is well known that the largest eigenvalue of a regular graph is its degree. However, it is generally difficult to determine the smallest eigenvalue of a regular graph. Recently, after having derived a recurrence formula (see Theorem \ref{Ren} below) for the eigenvalues of $\Gamma_{n}$, Renteln \cite{Renteln} showed that the smallest eigenvalue $\mu$ of $\Gamma_{n}$ is $-\frac{D_{n}}{n-1}$. The value of $\mu$ was also determined independently by Ellis et al. \cite{EFP} in their seminal work on intersecting families of permutations. The recurrence obtained by Renteln was later used by Ku and Wales \cite{Ku-Wales} to prove the Alternating Sign Property (ASP) of the derangement graph (Theorem \ref{alternating-sign}). The purpose of this paper is to give a new recurrence formula for these eigenvalues. This new recurrence, which follows from the property of shifted schur functions, provides a simpler proof of the ASP and settles affirmatively a conjecture of Ku and Wales regarding the lower bound and upper bounds for the absolute values of these eigenvalues.

Recall that a Cayley graph $\Gamma(G,S)$ is {\em normal} if $S$ is closed under conjugation. It is well known that the eigenvalues of a normal Cayley graph $\Gamma(G,S)$ can be expressed in terms of the irreducible characters of $G$.

\begin{thm}[\cite{Babai, DS, Lub, Ram}]\label{cayley}
The eigenvalues of a normal Cayley graph $\Gamma(G,S)$ are integers given by
\begin{eqnarray}
\eta_{\chi} & = & \frac{1}{\chi(1)} \sum_{s \in S} \chi(s),
\end{eqnarray}
where $\chi$ ranges over all the irreducible characters of $G$. Moreover, the multiplicity of $\eta_{\chi}$ is $\chi(1)^{2}$.
\end{thm}

Recall that a partition $\lambda$ of $n$, denoted by $\lambda \vdash n$, is a weakly decreasing sequence $\lambda_{1} \ge \ldots \ge \lambda_{r}$ with $\lambda_{r} \ge 1$ such that $\lambda_{1} + \cdots + \lambda_{r} = n$. We write $\lambda = (\lambda_{1}, \ldots, \lambda_{r})$. The {\em size} of $\lambda$, denoted by $|\lambda|$, is $n$ and each $\lambda_{i}$ is called the {\em $i$-th part} of the partition. We also use the notation $(\mu_{1}^{a_{1}}, \ldots, \mu_{s}^{a_{s}}) \vdash n$ to denote the partition where $\mu_{i}$ are the distinct nonzero parts that occur with multiplicity $a_{i}$. For example,
\[ (5,5,4,4,2,2,2,1) \longleftrightarrow (5^{2}, 4^{2}, 2^{3}, 1). \]

Clearly, the derangement graph $\Gamma_{n}$ is normal since the set $\mathcal{D}_{n}$ is closed under conjugation. On the other hand, it is well known that both the conjugacy classes of $\mathcal{S}_{n}$ and the irreducible characters of $\mathcal{S}_{n}$ are indexed by partitions $\lambda$ of $[n]$. Therefore, the eigenvalue $\eta_{\chi_{\lambda}}$ of the derangement graph can be denoted by $\eta_{\lambda}$. Throughout, we shall use this notation.

To describe the recurrence formula of Renteln, we require some terminology. To the Young diagram of a partition $\lambda$, we assign
$xy$-coordinates to each of its boxes by defining the
upper-left-most box to be $(1,1)$, with the $x$ axis increasing to
the right and the $y$ axis increasing downwards. Then the {\em hook}
of $\lambda$ is the union of the boxes $(x',1)$ and $(1, y')$ of the
Ferrers diagram of $\lambda$, where $x' \ge 1$, $y' \ge 1$. Let
$\w{h}_{\lambda}$ denote the hook of $\lambda$ and let $h_{\lambda}$
denote the size of $\w{h}_{\lambda}$. Similarly, let
$\w{c}_{\lambda}$ and $c_{\lambda}$ denote the first column of
$\lambda$ and the size of $\w{c}_{\lambda}$ respectively. Note that
$c_{\lambda}$ is equal to the number of rows of $\lambda$. When
$\lambda$ is clear from the context, we replace $\w{h}_{\lambda}$,
$h_{\lambda}$, $\w{c}_{\lambda}$ and $c_{\lambda}$ by $\w{h}$, $h$,
$\w{c}$ and $c$ respectively. Let $\lambda-\w{h} \vdash n-h$ denote
the partition obtained from $\lambda$ by removing its hook. Also,
let $\lambda-\w{c}$ denote the partition obtained from $\lambda$ by
removing the first column of its Ferrers diagram, i.e.
$(\lambda_{1}, \ldots, \lambda_{r})-\w{c} = (\lambda_{1}-1, \ldots,
\lambda_{r}-1) \vdash n-r$.

\begin{thm}[\cite{Renteln} Renteln's Formula]\label{Ren}
For any partition $\lambda$, the eigenvalues of the derangement graph $\Gamma_{n}$ satisfy the following recurrence:
\begin{eqnarray}
\eta_{\lambda} & = & (-1)^{h}\eta_{\lambda-\hat{h}} + (-1)^{h+\lambda_{1}}h\eta_{\lambda-\w{c}} \label{1-recurrence}
\end{eqnarray}
with initial condition $\eta_{\emptyset} = 1$.
\end{thm}

\begin{thm}[\cite{Ku-Wales} The Alternating Sign Property (ASP) ]\label{alternating-sign}
Let $n>1$. For any partition $\lambda = (\lambda_{1}, \ldots, \lambda_{r}) \vdash n$,
\begin{eqnarray}
\sign(\eta_{\lambda}) & = & (-1)^{|\lambda|-\lambda_{1}} \nonumber \\
& = & (-1)^{\# \textnormal{cells under the first row of $\lambda$}}
\end{eqnarray}
where $\sign(\eta_\lambda)$ is $1$ if $\eta_{\lambda}$ is positive or $-1$ if $\eta_{\lambda}$ is negative.
\end{thm}

It turns out that the two terms on the right-hand side of Renteln's formula (\ref{1-recurrence}) can have different signs. This is the source of difficulty in the proof of the ASP by Ku and Wales which relies mainly on the recurrence. Our recurrence formula does not have this problem, thus giving a `quicker' proof of the ASP.

To state our results, we need a new terminology. For a partition $\lambda =(\lambda_{1}, \ldots, \lambda_{r}) \vdash n$, let $\w{l_{\lambda}}$ denote the last row of $\lambda$ and let $l_{\lambda}$ denote the size of $\w{l_{\lambda}}$. Clearly, $l_{\lambda} = \lambda_{r}$. Also, let $\lambda-\w{l_{\lambda}}$ denote the partition obtained from $\lambda$ by deleting the last row. When $\lambda$ is clear from the context, we replace $\w{l_{\lambda}}$, $l_{\lambda}$ by $\w{l}$ and $l$ respectively.

\begin{thm}\label{formula}
Let $\lambda=(\lambda_{1}, \ldots, \lambda_{r}) \vdash n$. The eigenvalues of the derangement graph $\Gamma_{n}$ satisfy the following recurrence:
\begin{eqnarray}
\eta_{\lambda} & = & (-1)^{r-1}\lambda_{r} \eta_{\lambda-\hat{c}} + (-1)^{\lambda_{r}} \eta_{\lambda-\hat{l}} \label{main-recurrence}
\end{eqnarray}
with initial condition $\eta_{\emptyset} = 1$.
\end{thm}

It follows from the ASP that both of the terms on the right-hand side of (\ref{main-recurrence}) have the same sign.

Let $\lambda=(\lambda_{1}, \ldots, \lambda_{r}), \lambda'=(\lambda_{1}', \ldots, \lambda_{r}')\vdash n$. We write $\lambda <_{\textnormal{lex}} \lambda'$, if there is a $m$, $1\leq m\leq r$ such that $\lambda_i=\lambda_i'$ for all $1\leq i\leq m-1$ and $\lambda_m<\lambda_m'$. Note that `$<_{\textnormal{lex}}$' is the usual lexicographic ordering on the partitions of $n$.

Let $\lambda, \lambda'\vdash n$ with $\lambda_{1}$ as
their first part. In general, $\lambda <_{\textnormal{lex}} \lambda'$ does not imply that $\vert \eta_{\lambda}\vert<\vert \eta_{\lambda'}\vert$. This has been pointed out in \cite[Remark 1.4]{Ku-Wales}. One of our main contributions in this paper is to show that such property holds with respect to the {\em dominance order}. Recall that if $\lambda$ and $\lambda'$ are partitions, we say that $\lambda$ is {\em dominated} by $\lambda'$, and write $\lambda \unlhd \lambda'$, if $\lambda_{1} + \lambda_{2} + \cdots + \lambda_{k} \le \lambda'_{1} + \lambda'_{2} + \cdots + \lambda'_{k}$ for all positive integer $k$.

We give a more intuitive interpretation of the dominance order as follows. Recall that an {\em outside corner} of a partition $\lambda$ is a box $(x,y)$ of $\lambda$  such that neither $(x+1,y)$ nor $(x, y+1)$ are boxes of $\lambda$. On the other hand, define an {\em inside corner} of $\lambda$ as a location $(x,y)$ which is not a box of $\lambda$, such that either $y=1$ and $(x-1, y)$ is a box of $\lambda$, $x=1$ and $(x, y-1)$ is a box of $\lambda$, or $(x-1, y)$ and $(x, y-1)$ are boxes of $\lambda$. For example, in the following diagram of the partition $(4,3,1,1)$, the outside corners are marked with an `o' and the inside corners with an `i':

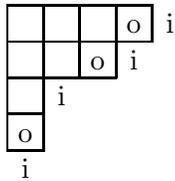
\begin{center}
\begin{young}
 &  &  &  o & ,i\\
 & &  o & ,i \\
    &, i \\
 o &, \\
 ,i
\end{young}
\end{center}
Let $\lambda, \lambda'\vdash n$. We write $\lambda <_1 \lambda'$, if there are $m_1$ and $m_2$, $1\leq m_1<m_2\leq r$ such that
\begin{align}
\lambda &=(\lambda_{1}, \ldots,\lambda_{m_1-1},\lambda_{m_1},\lambda_{m_1+1}, \ldots,\lambda_{m_2-1},\lambda_{m_2},\lambda_{m_2+1}, \ldots,\lambda_{r}),\notag\\
\lambda' &=(\lambda_{1}, \ldots,\lambda_{m_1-1},\lambda_{m_1}+1,\lambda_{m_1+1}, \ldots,\lambda_{m_2-1},\lambda_{m_2}-1,\lambda_{m_2+1}, \ldots,\lambda_{r})\notag
\end{align}
are partitions of $n$. Intuitively, $\lambda <_1 \lambda'$ corresponds to sliding
an outside corner of $\lambda$ upwards into an inside corner of $\lambda'$.

It turns out that the dominance order can be entirely characterized in terms of the partial ordering $<_1$. We shall omit the proof of this standard result.

\begin{lm}\label{raise}
Let $\mu$ and $\lambda$ be partitions of $n$. Then $\mu \unlhd \lambda$ if and only if there exist $\mu^{(1)}, \ldots, \mu^{(s)} \vdash n$ such that
\[ \mu <_1 \mu^{(1)} <_1 \cdots <_1 \mu^{(s)} <_1 \lambda. \]
\end{lm}

Using the recurrence given by Theorem \ref{formula}, we are able to prove Theorem \ref{thm_inequality} and then settle affirmatively the conjecture of Ku and Wales regarding the lower and upper bounds for the absolute values  of the eigenvalues of $\Gamma_{n}$ (Theorem \ref{Bounds}).

\begin{thm}\label{thm_inequality}
Let $\lambda, \lambda' \vdash n$ with $\lambda_{1}$ as
their first part. If $\lambda \unlhd \lambda'$, then
\begin{equation}
\vert \eta_{\lambda}\vert< \vert \eta_{\lambda'}\vert.\notag
\end{equation}
\end{thm}

\begin{thm}[The Ku-Wales Conjecture]\label{Bounds}
Suppose $\lambda^{*}  \vdash n$ is the largest partition in
lexicographic order among all the partitions with $\lambda_{1}$ as
their first part. Then, for every $\lambda = (\lambda_{1}, \ldots,
\lambda_{s}) \vdash n$,
\[|\eta_{(\lambda_{1}, 1^{n-\lambda_{1}})}| \le  |\eta_{\lambda}| \le |\eta_{\lambda^{*}}|. \]
\end{thm}

\begin{proof} It follows from Theorem \ref{thm_inequality} by noting that $(\lambda_{1}, 1^{n-\lambda_{1}}) \unlhd \lambda \unlhd \lambda^{*}$, for all $\lambda \vdash n$, $\lambda\neq \lambda^{*}, (\lambda_{1}, 1^{n-\lambda_{1}}) $.
\end{proof}
Note that it has been shown by Ku and Wales (see \cite[Theorem 1.3]{Ku-Wales}) that the lower bound holds for all $\lambda_{1} \ge \lfloor n/2 \rfloor$.

The paper is organized as follows. In Section 2, we introduce the shifted Schur functions developed by Okounkov and Olshanski \cite{OO} and rewrite a formula of Renteln in terms of these functions. Theorem \ref{formula} will then follow immediately from the property of these shifted Schur functions. Using the new recurrence formula, we provide a simpler proof of the ASP in Section 3. In Section 4, we proved Theorem \ref{thm_inequality}, thus settling a conjecture of Ku and Wales. For the reader's convenience, in Section 5, we reproduce some the eigenvalues of the derangement graphs for small $n$ as given in \cite{Ku-Wales}.

\section{Shifted Schur Functions}

The {\em Schur function} or {\em Schur polynomial} in $n$ variables can be defined as the ratio of two $n \times n$ determinants
\begin{eqnarray}
s_{\mu}(x_{1}, \ldots, x_{n})&  = &  \frac{\de \left[ x_{i}^{\mu_{j}+n-j}\right]}{\de \left[ x_{i}^{n-j}\right]},
\end{eqnarray}
where $\mu$ is an arbitrary partition $\mu_{1} \ge \mu_{2} \ge \cdots \mu_{n} \ge 0$ of length at most $n$.

An important variant of the Schur polynomial are the {\em shifted Schur polynomials}  that was developed by Okounkov and Olshanski \cite{OO}:
\begin{eqnarray}
s^{*}_{\mu}(x_{1}, \ldots, x_{n})& = & \frac{\de \left[ (x_{i}+n-i  \fall \mu_{j}+n-j ) \right]}{\de \left[ x_{i}+n-i \fall n-j \right]},
\end{eqnarray}
where the symbol $(x \fall k)$ is the $k$-th {\em falling factorial power} of a variable $x$:
\begin{eqnarray}
(x \fall k) & = & \left\{ \begin{array}{ll}
x(x-1)\cdots (x-k+1), &  \textnormal{if}~k=1,2, \ldots \\
1, & \textnormal{if}~k=0.
\end{array} \right.
\end{eqnarray}
Just like the ordinary Schur polynomials, the shifted Schur polynomials also satisfy the {\em stability property}:
\begin{eqnarray}
s^{*}_{\mu}(x_{1}, \ldots, x_{n},0) & = & s^{*}_{\mu}(x_{1}, \ldots, x_{n}). \label{stable}
\end{eqnarray}
The stability property allow us to define the {\em functions} $s^{*}_{\mu}(x_{1}, x_{2}, \ldots)$ in infinitely many variables that form a basis in the {\em algebra of shifted symmetric functions}, denoted by $\Lambda^{*}$. Every element of $\Lambda^{*}$ may be viewed as a function $f(x_{1}, x_{2}, \ldots)$ on an infinite sequence of arguments such that $x_{m} = 0$ for all sufficiently large $m$. We refer the reader to \cite{OO} for basic results on shifted symmetric functions.

For the application we have in mind, the following formula for the dimension of skew Young diagrams will be useful.

\begin{thm}[\cite{OO}] \label{dimension}
Let $\mu \vdash k$ and $\lambda \vdash n$ be two partitions, where $k \le n$ and $\mu \subseteq \lambda$. Let $\dimension \lambda/\mu$ denote the number of standard tableaux of shape $\lambda/\mu$; in particular, $\dimension \lambda = \dimension \lambda/\emptyset$. Then
\begin{eqnarray}
\frac{\dimension \lambda/\mu}{\dimension \lambda} & = & \frac{s^{*}_{\mu}(\lambda) }{(n \fall k)}, \label{dim}
\end{eqnarray}
where $s^{*}_{\mu}(\lambda) = s^{*}_{\mu}(\lambda_{1}, \lambda_{2}, \ldots)$.
\end{thm}

\begin{thm}[\cite{OO} Vanishing Theorem]\label{vanish}
We have
\begin{eqnarray}
s^{*}_{\mu}(\lambda) & = & 0~~\textnormal{unless}~~\mu \subseteq \lambda,  \\
s^{*}_{\mu}(\mu) & = & H(\mu),
\end{eqnarray}
where $H(\mu) = \prod_{\alpha \in \mu} h(\alpha)$ is the product of the hook lengths of all boxes of $\mu$.
\end{thm}

As an example of shifted symmetric functions, set $h^{*}_{k} = s^{*}_{(k)}$ where $(k)$ is the partition of $k$ whose Young diagram consists of just one row. These are called the {\em complete shifted symmetric functions}. They are shifted analogues of the complete homogeneous symmetric functions. We shall require the following properties of $h^{*}_{k}$.

\begin{prop}[\cite{OO}]\label{complete-shifted}
The complete shifted symmetric functions $h^{*}_{k}$ can be written as
\begin{eqnarray}
h^{*}_{k}(x_{1}, x_{2}, \ldots) & = & \sum_{1 \le i_{1} \le \cdots \le i_{k}< \infty} (x_{i_{1}}-k+1)(x_{i_{2}}-k+2) \cdots x_{i_{k}}. \label{complete}
\end{eqnarray}
\end{prop}

\begin{cor}
The complete shifted symmetric functions $h^{*}_{k}$ satisfy the following recurrence:
\begin{eqnarray}
h^{*}_{k}(x_{1}, \ldots, x_{n}) & = & x_{n} h^{*}_{k-1}(x_{1}-1, \ldots, x_{n}-1) + h^{*}_{k}(x_{1}, \ldots, x_{n-1}). \label{recurrence}
\end{eqnarray}
\end{cor}

\begin{proof}
In view of the stability property and Proposition \ref{complete-shifted}, we have
\[ h^{*}_{k}(x_{1}, \ldots, x_{n}) = \sum_{1 \le i_{1} \le \cdots \le i_{k} \le n } (x_{i_{1}}-k+1)(x_{i_{2}}-k+2) \cdots x_{i_{k}}. \]
Therefore,
\begin{eqnarray}
h^{*}_{k}(x_{1}, \ldots, x_{n})  & = & x_{n} \left(\sum_{1 \le i_{1} \le \cdots \le i_{k-1} \le n } (x_{i_{1}}-k+1)(x_{i_{2}}-k+2) \cdots (x_{i_{k}}-1) \right) \nonumber \\
 & & + \sum_{1 \le i_{1} \le \cdots \le i_{k} \le n-1 } (x_{i_{1}}-k+1)(x_{i_{2}}-k+2) \cdots x_{i_{k}} \nonumber \\
& = &  x_{n} h^{*}_{k-1}(x_{1}-1, \ldots, x_{n}-1) + h^{*}_{k}(x_{1}, \ldots, x_{n-1}). \nonumber
\end{eqnarray}
\end{proof}

Recall the following formula due to Renteln \cite[Theorem 3.2]{Renteln}.

\begin{thm}[\cite{Renteln}]\label{Renteln-2}
The eigenvalues of the derangement graph $\Gamma_{n}$ are given by
\begin{eqnarray}
\eta_{\lambda} & = & \sum_{k=0}^{n} (-1)^{n-k}(n \fall k) \frac{\dimension \lambda/(k)}{\dimension \lambda}
\end{eqnarray}
\end{thm}

\noindent Therefore, it follows immediately from Theorem \ref{dimension} and Theorem \ref{Renteln-2} that

\begin{cor}\label{shift}
The eigenvalues of the derangement graph $\Gamma_{n}$ are given by
\begin{eqnarray}
\eta_{\lambda} & = & \sum_{k=0}^{n} (-1)^{n-k}s^{*}_{(k)}(\lambda) \nonumber \\
& = & \sum_{k=0}^{n} (-1)^{n-k}h^{*}_{k}(\lambda).
\end{eqnarray}
\end{cor}


\noindent {\bf Proof of Theorem \ref{formula}.}\\[0.1cm]

\noindent Set  $\eta_{\lambda}' = \sum_{k=0}^{n} (-1)^{k} h^{*}_{k}(\lambda)$. By the Vanishing Theorem (Theorem \ref{vanish}) and Corollary \ref{shift}, we can write
\[ \eta_{\lambda}' = \sum_{k=0}^{\infty} (-1)^{k} h^{*}_{k}(\lambda) \]
so that
\[ \eta_{\lambda}' = (-1)^{|\lambda|}\eta_{\lambda}. \]
By (\ref{recurrence}),
\begin{eqnarray}
\eta_{\lambda}' & = &  \sum_{k=0}^{\infty} \left((-1)^{k} \left(\lambda_{r} h^{*}_{k-1}(\lambda_{1}-1, \ldots, \lambda_{r}-1) + h^{*}_{k}(\lambda_{1}, \ldots, \lambda_{r-1})\right) \right) \nonumber \\
& = & -\lambda_{r} \sum_{k=0}^{\infty} (-1)^{k-1}h^{*}_{k-1}(\lambda_{1}-1, \ldots, \lambda_{r}-1) + \sum_{k=0}^{\infty} (-1)^{k}h^{*}_{k}(\lambda_{1}, \ldots, \lambda_{r-1}) \nonumber \\
& = & -\lambda_{r} \eta_{\lambda-\hat{c}}' + \eta'_{\lambda-\hat{l}} \nonumber \\
& = & -\lambda_{r} (-1)^{|\lambda-\hat{c}|} \eta_{\lambda-\hat{c}} + (-1)^{|\lambda-\hat{l}|}\eta_{\lambda-\hat{l}} \nonumber \\
(-1)^{|\lambda|} \eta_{\lambda} & = & \lambda_{r}(-1)^{1+|\lambda|-r} \eta_{\lambda-\hat{c}} + (-1)^{|\lambda|-\lambda_{r}} \eta_{\lambda-\hat{l}} \nonumber \\
\eta_{\lambda} & = & (-1)^{r-1}\lambda_{r} \eta_{\lambda-\hat{c}} + (-1)^{\lambda_{r}} \eta_{\lambda-\hat{l}}.\notag
\end{eqnarray}

\hfill $\square$

\section{A simpler proof of the Alternating Sign Property}

We prove by induction on $|\lambda|$. Obviously, the property holds for all small partitions. By the inductive hypothesis,
\begin{eqnarray*}
\sign\left( (-1)^{r-1}\eta_{\lambda-\hat{c}}\right)&  = &  (-1)^{r-1} (-1)^{|\lambda-\hat{c}|-(\lambda_{1}-1)}\\
& = & (-1)^{r-1+|\lambda|-r-\lambda_{1}+1} \\
& = & (-1)^{|\lambda|-\lambda_{1}}.
\end{eqnarray*}
Similarly,
\begin{eqnarray*}
\sign \left( (-1)^{\lambda_{r}} \eta_{\lambda-\hat{l}} \right) & = & (-1)^{\lambda_{r}} (-1)^{|\lambda-\hat{l}|-\lambda_{1}} \\
& = & (-1)^{\lambda_{r} + |\lambda|-\lambda_{r}-\lambda_{1}} \\
& = & (-1)^{|\lambda|-\lambda_{1}}.
\end{eqnarray*}
By the recurrence formula (\ref{main-recurrence}), we deduce that
\[ \sign(\eta_{\lambda}) = (-1)^{|\lambda|-\lambda_{1}}. \]

\hfill $\square$

\section{Some preliminary lemmas}

For convenience, let us write
\begin{equation}
f(\lambda_1,\lambda_2,\dots, \lambda_r)=\vert \eta_{(\lambda_1,\lambda_2,\dots, \lambda_r)} \vert.\notag
\end{equation}
Then by Theorem \ref{alternating-sign} and Theorem \ref{formula}, we have
\begin{equation}
f(\lambda_1,\lambda_2,\dots, \lambda_r)=\lambda_{r}f(\lambda_1-1,\lambda_2-1,\dots, \lambda_r-1) + f(\lambda_1,\lambda_2,\dots, \lambda_{r-1}).\label{use_recur}
\end{equation}

By abuse of notation, in this section we shall use the symbol $\lambda$ to denote a positive integer instead of a partition.

\begin{lm}\label{lm_h_one_variable}
 \begin{align}
 h^{*}_{0}(\lambda)&=1,\notag\\
 h^{*}_{1}(\lambda)&=\lambda,\notag\\
  h^{*}_{k}(\lambda)&=(\lambda-k+1)(\lambda-k+2)\cdots(\lambda-1)(\lambda),\ \ \textnormal{for $k\geq 2$}.\notag
\end{align}
\end{lm}

\begin{proof} It follows easily from Proposition \ref{complete-shifted}.
\end{proof}

\begin{lm}\label{recurrence_general} For any $1< m\leq r$,
\begin{equation}
f(\lambda_1,\lambda_2,\dots, \lambda_r)=\sum_{k=0}^{\lambda_m} h^{*}_{k}(\lambda_m,\dots, \lambda_r) f(\lambda_1-k,\lambda_2-k,\dots, \lambda_{m-1}-k).\notag
\end{equation}
\end{lm}

\begin{proof} Repeatedly applying equation (\ref{use_recur}) and by Lemma \ref{lm_h_one_variable}, we obtain
\begin{align}
&f(\lambda_1,\lambda_2,\dots, \lambda_r)\notag\\
&=h^{*}_{1}(\lambda_r)f(\lambda_1-1,\lambda_2-1,\dots, \lambda_{r}-1) + h^{*}_{0}(\lambda_r)f(\lambda_1,\lambda_2,\dots, \lambda_{r-1})\notag\\
&=(\lambda_{r})(\lambda_r-1)f(\lambda_1-2,\lambda_2-2,\dots, \lambda_r-2)+\sum_{k=0}^{1} h^{*}_{k}(\lambda_r) f(\lambda_1-k,\lambda_2-k,\dots, \lambda_{r-1}-k)\notag\\
&=(\lambda_{r})(\lambda_r-1)(\lambda_r-2)f(\lambda_1-3,\lambda_2-3,\dots, \lambda_r-3)+\sum_{k=0}^{2} h^{*}_{k}(\lambda_r) f(\lambda_1-k,\lambda_2-k,\dots, \lambda_{r-1}-k)\notag\\
&\hskip 3cm\vdots\notag\\
&=\sum_{k=0}^{\lambda_r} h^{*}_{k}(\lambda_r) f(\lambda_1-k,\lambda_2-k,\dots, \lambda_{r-1}-k).\label{eq_use}
\end{align}
Thus the lemma holds for $m=r$. Assume that it holds for some $m_0$, $2<m_0\leq r$. We shall show that it also holds for $m_0-1$.

By assumption, the following equation holds:
\begin{equation}
f(\lambda_1,\lambda_2,\dots, \lambda_r)=\sum_{k=0}^{\lambda_{m_0}} h^{*}_{k}(\lambda_{m_0},\dots, \lambda_r) f(\lambda_1-k,\lambda_2-k,\dots, \lambda_{m_0-1}-k).\label{eq_use2}
\end{equation}
By applying equation (\ref{eq_use}),
\begin{align}
&f(\lambda_1,\lambda_2,\dots, \lambda_r)\notag\\
&=\sum_{k=0}^{\lambda_{m_0}} h^{*}_{k}(\lambda_{m_0},\dots, \lambda_r)
\left (\sum_{j=0}^{\lambda_{m_0-1}-k} h^{*}_{j}(\lambda_{m_0-1}-k) f(\lambda_1-k-j,\lambda_2-k-j,\dots, \lambda_{m_0-2}-k-j) \right)\notag\\
&=\sum_{k=0}^{\lambda_{m_0}} \sum_{j=0}^{\lambda_{m_0-1}-k} h^{*}_{k}(\lambda_{m_0},\dots, \lambda_r)
 h^{*}_{j}(\lambda_{m_0-1}-k) f(\lambda_1-k-j,\lambda_2-k-j,\dots, \lambda_{m_0-2}-k-j).\label{eq_use3}
\end{align}
Now by collecting all the terms with $k+j=j_0$, equation (\ref{eq_use3}) becomes
\begin{align}
&f(\lambda_1,\lambda_2,\dots, \lambda_r)\notag\\
&=\sum_{j_0=0}^{\lambda_{m_0-1}} \left( \sum_{\substack{k+j=j_0,\\ 0\leq k\leq \lambda_{m_0}}} h^{*}_{j}(\lambda_{m_0-1}-k)h^{*}_{k}(\lambda_{m_0},\dots, \lambda_r)
  \right) f(\lambda_1-j_0,\lambda_2-j_0,\dots, \lambda_{m_0-2}-j_0).\label{eq_use4}
\end{align}
By Proposition \ref{complete-shifted},
\begin{equation}
 h^{*}_{j_0}(\lambda_{m_0-1},\lambda_{m_0},\dots, \lambda_r)=\sum_{\substack{k+j=j_0,\\ 0\leq k\leq \lambda_{m_0}}} h^{*}_{j}(\lambda_{m_0-1}-k)h^{*}_{k}(\lambda_{m_0},\dots, \lambda_r).\notag
\end{equation}
Thus, by induction the lemma follows.
\end{proof}

\begin{lm}\label{inequality_h0}
 \begin{align}
 h^{*}_{0}(\lambda_1,\lambda_2,\dots,\lambda_r)&=1,\notag\\
 h^{*}_{1}(\lambda_1,\lambda_2,\dots,\lambda_r)&=\lambda_1+\lambda_2+\cdots+\lambda_r.\notag
\end{align}
\end{lm}

\begin{proof} It follows easily from Proposition \ref{complete-shifted}.
\end{proof}

\begin{lm}\label{inequality_h1}
If $\lambda_s\leq \lambda$ and $2\leq k\leq \lambda$, then
 \begin{equation}
 h^{*}_{k}(\lambda,\lambda_s)< h^{*}_{k}(\lambda+1,\lambda_s-1).\notag
\end{equation}
\end{lm}

\begin{proof} By Proposition \ref{complete-shifted},
\begin{equation}
h^{*}_{k}(x,y)=\sum_{j=0}^k   (x-j\fall k-j)(y \fall j).\notag
\end{equation}
 Therefore,
 \begin{equation}
h^{*}_{k}(\lambda+1,\lambda_s-1)-h^{*}_{k}(\lambda,\lambda_s-1)=\sum_{j=0}^{k-1} (k-j)  (\lambda-j\fall k-j-1)(\lambda_s-1 \fall j),\label{compare2}
\end{equation}
 and
 \begin{align}
h^{*}_{k}(\lambda,\lambda_s)-h^{*}_{k}(\lambda,\lambda_s-1) &=\sum_{j=1}^{k} j  (\lambda-j\fall k-j)(\lambda_s-1 \fall j-1).\notag\\
&=\sum_{j=0}^{k-1} (j+1)  (\lambda-j-1\fall k-j-1)(\lambda_s-1 \fall j).\label{compare1}
\end{align}

We shall compare equation (\ref{compare2}) with equation (\ref{compare1}). For $0\leq j<\frac{k-1}{2}$, the $j$-th and $(k-1-j)$-th term of the right side of equation (\ref{compare2}) are
\begin{align}
& (k-j)  (\lambda-j\fall k-j-1)(\lambda_s-1 \fall j),\label{eq_first1}\\
& (j+1)  (\lambda-k+1+j\fall j)(\lambda_s-1 \fall k-1-j).\label{eq_first2}
\end{align}
On the other hand, the $j$-th and $(k-1-j)$-th term of the right side of equation (\ref{compare1}) are
\begin{align}
&  (j+1)  (\lambda-j-1\fall k-j-1)(\lambda_s-1 \fall j),\label{eq_second1}\\
&  (k-j)  (\lambda-k+j\fall j)(\lambda_s-1 \fall k-1-j).\label{eq_second2}
\end{align}

When $j=0$, the sum (\ref{eq_first1}) $+$ (\ref{eq_first2}) $-$ (\ref{eq_second1}) $-$ (\ref{eq_second2}) is
 \begin{align}
& k  (\lambda\fall k-1)+(\lambda_s-1 \fall k-1)- (\lambda-1\fall k-1)-k(\lambda_s-1 \fall k-1)\notag\\
 & = ((\lambda\fall k-1)-(\lambda-1\fall k-1))+(k-1)((\lambda\fall k-1)-(\lambda_s-1 \fall k-1))\notag\\
  & = ((k-1)(\lambda-1\fall k-2))+(k-1)((\lambda\fall k-1)-(\lambda_s-1 \fall k-1))\notag\\
  &>0,\label{eq_final1}
 \end{align}
 where the last inequality follows from $k\geq 2$ and $\lambda\geq \lambda_{s}>\lambda_s-1$.

 Now for $1\leq j<\frac{k-1}{2}$,  (\ref{eq_first1}) $-$ (\ref{eq_second1}) is
\begin{equation}
\left((k-j)(\lambda-j)-(j+1)(\lambda-k+1)\right)(\lambda-j-1\fall k-j-2)(\lambda_s-1 \fall j),\label{eq_third1}
\end{equation}
 and (\ref{eq_first2}) $-$ (\ref{eq_second2}) is
 \begin{equation}
\left((j+1)(\lambda-k+1+j)-(k-j)(\lambda-k+1)\right)(\lambda-k+j\fall j-1)(\lambda_s-1 \fall k-1-j).\label{eq_third2}
\end{equation}
Since $\lambda\geq \lambda_s$ and $j<\frac{k-1}{2}$,
\begin{align}
(k-j)(\lambda-j)-(j+1)(\lambda-k+1) &=(k-(2j+1))\lambda+(k-1)+j(j-1)>0,\notag\\
(\lambda-j-1\fall k-j-2)(\lambda_s-1 \fall j)&\geq (\lambda-k+j\fall j-1)(\lambda_s-1 \fall k-1-j).\notag
\end{align}
Therefore the sum (\ref{eq_third1})+(\ref{eq_third2}) is at least
\begin{align}
((k-j)(k-j-1)+(j+1)j)(\lambda-k+j\fall j-1)(\lambda_s-1 \fall k-1-j)>0.\label{eq_final2}
\end{align}

If $k$ is odd, then $j$ can take value $\frac{k-1}{2}$. The $\frac{k-1}{2}$-th term on the right side of (\ref{compare2})   is
\begin{eqnarray}
\frac{k+1}{2}  \left(\lambda-\frac{k-1}{2} \left \fall  \frac{k-1}{2} \right. \right)\left(\lambda_s-1 \left \fall \frac{k-1}{2}\right.\right),\label{eq_fourth1}
\end{eqnarray}
and  the $\frac{k-1}{2}$-th term on the right side of (\ref{compare1})  is
\begin{equation}
 \frac{k+1}{2}  \left(\lambda-\frac{k+1}{2} \left \fall \frac{k-1}{2} \right.\right)\left(\lambda_s-1 \left \fall  \frac{k-1}{2} \right.\right).\label{eq_fourth2}
\end{equation}
Note that (\ref{eq_fourth1}) $-$ (\ref{eq_fourth2})  is
\begin{equation}
\frac{k+1}{2}\frac{k-1}{2}  \left(\lambda-\frac{k+1}{2} \left \fall \frac{k-1}{2}-1 \right.\right)\left(\lambda_s-1 \left \fall  \frac{k-1}{2} \right.\right)> 0.\label{eq_final3}
\end{equation}

From equations (\ref{eq_final1}), (\ref{eq_final2}) and (\ref{eq_final3}), we deduce that
\begin{align}
& h^{*}_{k}(\lambda+1,\lambda_s-1)- h^{*}_{k}(\lambda,\lambda_s)\notag\\
&=\left(h^{*}_{k}(\lambda+1,\lambda_s-1)-h^{*}_{k}(\lambda,\lambda_s-1) \right)\notag\\
&\hskip 2cm -\left(h^{*}_{k}(\lambda,\lambda_s)-h^{*}_{k}(\lambda,\lambda_s-1)\right)\notag\\
&> 0.\notag
\end{align}
\end{proof}

\begin{lm}\label{inequality_h2} Let $l\geq 1$ and
\begin{equation}
 h^{*}_{k}(\lambda,\lambda^l,\lambda)=h^{*}_{k}(\lambda,\underbrace{\lambda,\dots,\lambda}_{l\ \textnormal{times}},\lambda).\notag
\end{equation}
If $2\leq k\leq \lambda$, then
 \begin{equation}
 h^{*}_{k}(\lambda,\lambda^l,\lambda)< h^{*}_{k}(\lambda+1,\lambda^l,\lambda-1).\notag
\end{equation}
\end{lm}

\begin{proof} By Proposition \ref{complete-shifted},
\begin{align}
&(x-j-r \fall k-j-r)(\lambda-j\fall r)(y \fall j)\notag\\
&=(x-k+1)\cdots (x-j-r)(\lambda-j-r+1)\cdots (\lambda-j)(y-j+1)\cdots (y),\notag
\end{align}
is a term in the sum of $h^{*}_{k}(x,\lambda^l,y)$. In fact, there are $\binom{r+l-1}{l-1}$ such terms. Therefore
\begin{equation}
h^{*}_{k}(x,\lambda^l,y)=\sum_{j=0}^k \sum_{r=0}^{k-j}\binom{r+l-1}{l-1}  (x-j-r \fall k-j-r)(\lambda-j\fall r)(y \fall j).\label{leq_first}
\end{equation}
From (\ref{leq_first}),
\begin{align}
&h^{*}_{k}(x+1,\lambda^l,y)-h^{*}_{k}(x,\lambda^l,y)\notag\\
&= \sum_{j=0}^{k-1} \sum_{r=0}^{k-1-j}(k-j-r)\binom{r+l-1}{l-1}  (x-j-r \fall k-1-j-r)(\lambda-j\fall r)(y \fall j).\label{leq_difference}
\end{align}
Now replacing $x$ with $\lambda$ and $y$ with $\lambda-1$ in (\ref{leq_difference}), we obtain
\begin{align}
&h^{*}_{k}(\lambda+1,\lambda^l,\lambda-1)-h^{*}_{k}(\lambda,\lambda^l,\lambda-1)\notag\\
&= \sum_{j=0}^{k-1} \sum_{r=0}^{k-1-j}(k-j-r)\binom{r+l-1}{l-1}  (\lambda-j\fall k-1-j)(\lambda-1 \fall j)\notag\\
&=\sum_{j=0}^{k-1} \binom{k-j+l}{l+1}  (\lambda-j\fall k-1-j)(\lambda-1 \fall j)\notag\\
&=\sum_{j=0}^{k-1} \binom{k-j+l}{l+1}(\lambda-j)  (\lambda-1 \fall k-2)\notag\\
&> \sum_{j=0}^{k-1} \binom{k-j+l}{l+1}(\lambda-k+1)  (\lambda-1 \fall k-2)\notag\\
&=\sum_{j=0}^{k-1} \binom{k-j+l}{l+1}(\lambda-1 \fall k-1)\notag\\
&= \binom{k+l+1}{l+2} (\lambda-1\fall k-1).\label{leq_compare_1}
\end{align}
From (\ref{leq_first}),
\begin{align}
&h^{*}_{k}(x,\lambda^l,y)-h^{*}_{k}(x,\lambda^l,y-1)\notag\\
&= \sum_{j=1}^{k} \sum_{r=0}^{k-j}j\binom{r+l-1}{l-1}  (x-j-r \fall k-j-r)(\lambda-j\fall r)(y-1\fall j-1).\label{leq_difference2}
\end{align}
Now replacing $x$ with $\lambda$ and $y$ with $\lambda$ in (\ref{leq_difference2}), we obtain
\begin{align}
&h^{*}_{k}(\lambda,\lambda^l,\lambda)-h^{*}_{k}(\lambda,\lambda^l,\lambda-1)\notag\\
&=  \sum_{j=1}^{k} \sum_{r=0}^{k-j}j\binom{r+l-1}{l-1}  (\lambda-j\fall k-j)(\lambda-1\fall j-1)\notag\\
&=  \sum_{j=1}^{k} \sum_{r=0}^{k-j}j\binom{r+l-1}{l-1} (\lambda-1\fall k-1)\notag\\
&= \sum_{j=1}^{k} j\binom{k-j+l}{l} (\lambda-1\fall k-1)\notag\\
&= \binom{k+l+1}{l+2} (\lambda-1\fall k-1).\label{leq_compare_2}
\end{align}
By equations (\ref{leq_compare_1}) and (\ref{leq_compare_2}), we deduce that
\begin{equation}
h^{*}_{k}(\lambda+1,\lambda^l,\lambda-1)-h^{*}_{k}(\lambda,\lambda^l,\lambda)>0.\notag
\end{equation}
\end{proof}

\begin{lm}\label{lm_special_1} Let $r\geq 3$. If
\begin{align}
&(\lambda_1,\dots, \lambda_{r-2},\lambda_{r-1},\lambda_r),\notag\\
&(\lambda_1,\dots, \lambda_{r-2},\lambda_{r-1}+1,\lambda_r-1),\notag
\end{align}
are two partitions of $n$, then
\begin{equation}
f(\lambda_1,\dots, \lambda_{r-2},\lambda_{r-1},\lambda_r)<f(\lambda_1,\dots, \lambda_{r-2},\lambda_{r-1}+1,\lambda_r-1).\notag
\end{equation}
\end{lm}

\begin{proof} By Lemma \ref{recurrence_general},
\begin{equation}
f(\lambda_1,\dots, \lambda_{r-2},\lambda_{r-1},\lambda_r)=\sum_{k=0}^{\lambda_{r-1}} h^{*}_{k}(\lambda_{r-1}, \lambda_r) f(\lambda_1-k,\lambda_2-k,\dots, \lambda_{r-2}-k),\notag
\end{equation}
and
\begin{align}
f(\lambda_1,\dots, \lambda_{r-2},\lambda_{r-1}+1,\lambda_r-1)&=\sum_{k=0}^{\lambda_{r-1}+1} h^{*}_{k}(\lambda_{r-1}+1, \lambda_r-1) f(\lambda_1-k,\lambda_2-k,\dots, \lambda_{r-2}-k)\notag\\
&\geq \sum_{k=0}^{\lambda_{r-1}} h^{*}_{k}(\lambda_{r-1}+1, \lambda_r-1) f(\lambda_1-k,\lambda_2-k,\dots, \lambda_{r-2}-k).\notag
\end{align}
The lemma then follows from Lemma \ref{inequality_h0} and Lemma \ref{inequality_h1}.
\end{proof}

\begin{lm}\label{lm_special_2} If  $l\geq 1$ and
\begin{align}
&(\lambda_1,\dots, \lambda_{r},\lambda,\lambda^l,\lambda),\notag\\
&(\lambda_1,\dots, \lambda_{r},\lambda+1,\lambda^l,\lambda-1),\notag
\end{align}
are two partitions of $n$, then
\begin{equation}
f(\lambda_1,\dots, \lambda_{r},\lambda,\lambda^l,\lambda)<f(\lambda_1,\dots, \lambda_{r},\lambda+1,\lambda^l,\lambda-1).\notag
\end{equation}
\end{lm}

\begin{proof} By Lemma \ref{recurrence_general},
\begin{equation}
f(\lambda_1,\dots, \lambda_{r},\lambda,\lambda^l,\lambda)=\sum_{k=0}^{\lambda} h^{*}_{k}(\lambda,\lambda^l,\lambda) f(\lambda_1-k,\lambda_2-k,\dots, \lambda_{r}-k),\notag
\end{equation}
and
\begin{align}
f(\lambda_1,\dots, \lambda_{r},\lambda+1,\lambda^l,\lambda-1)&=\sum_{k=0}^{\lambda+1} h^{*}_{k}(\lambda+1,\lambda^l,\lambda-1) f(\lambda_1-k,\lambda_2-k,\dots, \lambda_{r}-k)\notag\\
&\geq\sum_{k=0}^{\lambda} h^{*}_{k}(\lambda+1,\lambda^l,\lambda-1) f(\lambda_1-k,\lambda_2-k,\dots, \lambda_{r}-k).\notag
\end{align}
The lemma then follows from Lemma \ref{inequality_h0} and Lemma \ref{inequality_h2}.
\end{proof}

\section{Proof of Theorem \ref{thm_inequality}}

\begin{proof} By Lemma \ref{raise}, it is sufficient to show that the inequality holds for $\lambda<_1\lambda'$, i.e. if $\lambda <_1 \lambda'$ then $|\eta_{\lambda}| < |\eta_{\lambda'}|$.
 
 Let $2\leq m_1<m_2\leq r$ be such that
\begin{align}
\lambda & =(\lambda_1,\dots,\lambda_{m_1-1},\lambda_{m_1},\lambda_{m_1+1},\dots,  \lambda_{m_2-1},\lambda_{m_2},\lambda_{m_2+1},\dots, \lambda_{r})\notag\\
\lambda' & =(\lambda_1,\dots,\lambda_{m_1-1},\lambda_{m_1}+1,\lambda_{m_1+1},\dots , \lambda_{m_2-1},\lambda_{m_2}-1,\lambda_{m_2+1},\dots, \lambda_{r}).\notag
\end{align}
We shall prove by induction on $n$. Clearly, Theorem \ref{thm_inequality} holds for small values of $n$. We shall distinguish two cases.

\vskip 0.5cm
\noindent
{\bf Case 1.} $m_2\neq r$. Then by Theorem \ref{alternating-sign} and Theorem \ref{formula},
\begin{eqnarray}
\vert \eta_{\lambda}\vert & = & \lambda_{r} \vert \eta_{\lambda-\hat{c}}\vert + \vert \eta_{\lambda-\hat{l}}\vert.\notag
\end{eqnarray}
Note that   $\lambda-\hat{c}<_1\lambda'-\hat{c}$ and $\lambda-\hat{l}<_1\lambda'-\hat{l}$. So, by induction,
\begin{eqnarray}
\vert \eta_{\lambda}\vert & = & \lambda_{r} \vert \eta_{\lambda-\hat{c}}\vert + \vert \eta_{\lambda-\hat{l}}\vert<\lambda_{r} \vert \eta_{\lambda'-\hat{c}}\vert + \vert \eta_{\lambda'-\hat{l}}\vert=\vert \eta_{\lambda'}\vert.\notag
\end{eqnarray}

\vskip 0.5cm
\noindent
{\bf Case 2.} $m_2=r$. If $m_1=r-1$, then it follows from Lemma \ref{lm_special_1} that $\vert \eta_{\lambda}\vert<\vert \eta_{\lambda'}\vert$. Suppose $m_1<r-1$.

Let $m_3$ be the largest integer such that
\begin{equation}
\lambda''  =(\lambda_1,\dots,\lambda_{m_3-1},\lambda_{m_3}+1,\lambda_{m_3+1},\dots,\lambda_{r}-1),\notag
\end{equation}
is a partition of $n$. Note that $m_1\leq m_3$. By the choice of $m_3$, we must have
\begin{equation}
\lambda_{m_3-1}>\lambda_{m_3}=\lambda_{m_3+1}=\cdots=\lambda_{r-1}.\notag
\end{equation}
If $\lambda_r=\lambda_{r-1}$, then by Lemma \ref{lm_special_2}, $\vert \eta_{\lambda}\vert<\vert \eta_{\lambda''}\vert$. If $\lambda_r<\lambda_{r-1}$, then by Case 1, $\vert \eta_{\lambda}\vert<\vert \eta_{\lambda'''}\vert$, where
\begin{equation}
\lambda'''  =(\lambda_1,\dots,\lambda_{m_3-1},\lambda_{m_3}+1,\lambda_{m_3+1},\dots,\lambda_{r-1}-1,\lambda_r).\notag
\end{equation}
By Lemma \ref{lm_special_1}, $\vert \eta_{\lambda'''}\vert<\vert \eta_{\lambda''}\vert$. Thus $\vert \eta_{\lambda}\vert<\vert \eta_{\lambda''}\vert$.

In either case, $\vert \eta_{\lambda}\vert<\vert \eta_{\lambda''}\vert$. If $m_1=m_3$, then we are done. If $m_1<m_3$, then by Case 1, $\vert \eta_{\lambda''}\vert<\vert \eta_{\lambda'}\vert$. Hence $\vert \eta_{\lambda}\vert<\vert \eta_{\lambda'}\vert$. This completes the proof of the theorem.
\end{proof}

\begin{section}{Some Values of $\eta_\lambda$}\label{values}

In this section we reproduce some of the eigenvalues of $\Gamma_{n}$ for small $n$ as given in \cite{Ku-Wales}.

\medskip


{\small
\begin{center}
\begin{minipage}[t]{5cm}
\begin{tabular}{|rr|}
\multicolumn{2}{c} {$n=2$}  \\
\hline
$\lambda$ &  $\eta_{\lambda}$  \\
\hline
$2$& $1$  \\
$1^{2}$& $-1$  \\
\hline
\end{tabular}
\end{minipage}
\hspace{2cm}
\begin{minipage}[t]{5cm}
\begin{tabular}{|rr|}
\multicolumn{2}{c} {$n=3$} \\
\hline
$\lambda$ &  $\eta_{\lambda}$  \\
\hline
$3$& $2$ \\
$2,1$& $-1$ \\
$1^{3}$& $2$ \\
\hline
\end{tabular}
\end{minipage}
\end{center}


\begin{center}
\begin{minipage}[t]{5cm}
\begin{tabular}{|rr|r|rr|}
\multicolumn{5}{c} {$n=4$} \\
\hline
$\lambda$ &  $\eta_{\lambda}$ & \ \ \ & $\lambda$ &  $\eta_{\lambda}$  \\
\hline
$4$& $9$ & & $2,1^{2}$ & $1$ \\
$3,1$& $-3$ & &  $1^{4}$ & $-3$ \\
$2,2$& $3$ & &  &  \\
\hline
\end{tabular}
\end{minipage}
\hspace{2cm}
\begin{minipage}[t]{5cm}
\begin{tabular}{|rr|r|rr|}
\multicolumn{5}{c} {$n=5$} \\
\hline
$\lambda$ &  $\eta_{\lambda}$ & \ \ \ & $\lambda$ &  $\eta_{\lambda}$  \\
\hline
$5$& $44$ & & $2^{2},1$ & $-4$ \\
$4,1$& $-11$ & & $2,1^{3}$ & $-1$ \\
$3,2$& $4$ & & $1^{5}$ & $4$ \\
$3,1^{2}$& $4$ & &  &  \\
\hline
\end{tabular}
\end{minipage}
\end{center}


\begin{center}
\begin{minipage}[t]{5cm}
\begin{tabular}{|rr|r|rr|}
\multicolumn{5}{c} {$n=6$} \\
\hline
$\lambda$ &  $\eta_{\lambda}$ & \ \ \ & $\lambda$ &  $\eta_{\lambda}$  \\
\hline
$ 6 $  & $265$ & & $ 3, 1^{3}$ & $-5$\\
$ 5, 1 $  & $-53$ & & $ 2^{3} $ & $7$\\
$ 4, 2 $ & $15$ & & $ 2^{2}, 1^{2} $ & $5$\\
$4, 1^{2} $ & $13$ & & $ 2, 1^{4} $ & $ 1$\\
$3^{2}$ & $-11$ & & $ 1^{6} $ & $-5$ \\
$3, 2, 1$ & $-5$ & & &\\
\hline
\end{tabular}
\end{minipage}
\hspace{2cm}
\begin{minipage}[t]{5cm}
\begin{tabular}{|rr|r|rr|}
\multicolumn{5}{c} {$n=7$} \\
\hline
$\lambda$ &  $\eta_{\lambda}$ & \ \ \ & $\lambda$ &  $\eta_{\lambda}$  \\
\hline
$ 7 $ & $1854$  & & $3, 2^{2} $ & $ 6$\\
$6, 1 $ & $-309$ & & $3, 2, 1^{2}$ & $6$\\
$5, 2 $ & $66$ & & $3, 1^{4} $ & $ 6$\\
$5, 1, 1$ & $62$ & &$2^{3}, 1 $ & $-9$ \\
$4, 3 $ & $-21$ & &$2^{2}, 1^{3} $ & $-6$ \\
$4, 2, 1$ & $-18$ & &$2, 1^{5}$ & $ -1$ \\
$4, 1^{3} $ & $ -15$ & &$1^{7}$ & $ 6$ \\
$3^{2}, 1 $ & $ 14$ & & & \\
\hline
\end{tabular}
\end{minipage}
\end{center}


\begin{center}
\begin{minipage}[t]{5cm}
\begin{tabular}{|rr|r|rr|}
\multicolumn{5}{c} {$n=8$} \\
\hline
$\lambda$ &  $\eta_{\lambda}$ & \ \ \ & $\lambda$ &  $\eta_{\lambda}$  \\
\hline
$ 8 $ & $ 14833$ & & $4, 1^{4}$ & $ 17$\\
$7, 1$ &  $-2119$ &&$3^{2}, 2$ & $-19$ \\
$6, 2$ & $371$ & & $3^{2}, 1^{2}$ & $ -17$\\
$6, 1^{2}$ & $353$ & & $3, 2^{2}, 1$ & $-7$\\
$5, 3$ & $-89$ & & $3, 2, 1^{3}$ & $-7$\\
$5, 2, 1$ & $-77$ & & $3, 1^{5}$ & $-7$\\
$5, 1^{3}$ & $ -71$ & & $2^{4}$ & $13$\\
$4^{2}$ & $53$ & &$2^{3}, 1^{2}$ & $11$ \\
$4, 3, 1$ & $25$ & & $2^{2}, 1^{4}$ & $7$\\
$4, 2^{2}$ & $ 23$ & & $2, 1^{6}$ & $1$ \\
$4, 2, 1^{2}$ & $ 21$ & &$1^{8}$ & $-7$ \\
\hline
\end{tabular}
\end{minipage}
\hspace{2cm}
\begin{minipage}[t]{5cm}
\begin{tabular}{|rr|r|rr|}
\multicolumn{5}{c} {$n=9$} \\
\hline
$\lambda$ &  $\eta_{\lambda}$ & \ \ \ & $\lambda$ &  $\eta_{\lambda}$  \\
\hline
$9$ & $ 133496$ &&$4, 2^{2}, 1$ & $  -27$ \\
$8, 1$ & $-16687$ && $4, 2, 1^{3}$ & $  -24$\\
$7, 2$ & $2472$ &&$4, 1^{5}$ & $  -19$ \\
$7, 1^{2}$ & $2384$ && $3^{3}$ & $  32$\\
$6, 3$ & $ -463$ &&$3^{2}, 2, 1$ & $ 23$ \\
$6, 2, 1$ & $-424$ && $3^{2}, 1^{3}$ & $  20$\\
$6, 1^{3}$ & $-397$ &&$3, 2^{3}$ & $    8$ \\
$5, 4$ & $128$ && $3, 2^{2}, 1^{2}$ & $  8$\\
$5, 3, 1$ & $104$ && $3, 2, 1^{4}$ & $  8$\\
$5, 2^{2}$ & $92$ &&$3, 1^{6}$ & $  8$ \\
$5, 2, 1^{2}$ & $ 88$ &&$2^{4}, 1$ & $  -16$ \\
$5, 1^{4}$ & $  80$ && $2^{3}, 1^{3}$ & $  -13$\\
$4^{2}, 1$ & $  -64$ &&$2^{2}, 1^{5}$ & $  -8$ \\
$4, 3, 2$ & $  -31$ && $2, 1^{7}$ & $ -1$\\
$4, 3, 1^{2}$ & $  -29$ && $1^{9}$ & $  8$\\
\hline
\end{tabular}
\end{minipage}
\end{center}

\begin{center}
\begin{tabular}{|rr|r|rr|r|rr|r|rr|}
\multicolumn{11}{c} {$n=10$} \\
\hline
$\lambda$ &  $\eta_{\lambda}$ & \ \ \ & $\lambda$ &  $\eta_{\lambda}$ & \ \ \  &  $\lambda$ & $\eta_{\lambda}$&  \ \ \ &$\lambda$ & $\eta_{\lambda}$ \\
\hline
$ 10  $ & $  1334961 $ & & $  6,1^4  $ & $ 441  $ & & $ 4,3,2,1$ & $36$ & &$ 3,2^2,1^3$ & $-9$ \\
$ 9,1   $ & $ -148329  $ & & $5,5 $ & $-309 $  & & $4,3,1^3 $ & $33$& &$3,2,1^5$  &$-9$ \\
$  8,2  $ & $  19071 $ & &  $5,4,1  $ & $ -149$  & & $4,2^3 $ & $33$& &$3,1^7$  &$-9$ \\
$  8,1^2  $ & $  18541 $ & & $5,3,2 $ & $-125 $  & & $4,2^2,1^2 $ & $31$& &$2^5$  &$21$ \\
$  7,3  $ & $   -2967$ & & $5.3.1^2 $ & $-119 $  & & $4,2,1^4 $ & $27$& &$2^4,1^2$  &$19$ \\
$ 7,2,1   $ & $  -2781$ & & $5,2^2,1 $ & $-105 $  & & $4,1^6 $ & $21$& &$2^3,1^4$  &$15$ \\
$ 7,1^3   $ & $ -2649  $ & & $5,2,1^3 $ & $-99 $  & & $3^3,1 $ & $-39$& &$2^2,1^6$  &$9$ \\
$  6,4  $ & $  621 $ & & $5,1^5 $ & $-89 $  & & $3^2,2^2 $ & $-29$& & $2,1^8$ &$1$  \\
$ 6,3,1  $ & $529   $ & & $4^2,2 $ & $81 $  & & $3^2,2,1^2 $ & $-27$& & $1^{10}$ &$-9$ \\
$ 6,2^2   $ & $495   $ & & $4^2,1^2 $ & $75 $  & & $3^2,1^4 $ & $-23$& &  & \\
$ 6,2,1^2   $ & $  477 $ & & $4,3^2 $ & $39 $  & & $3,2^3,1 $ & $-9$& &  & \\
\hline
\end{tabular}
\end{center}
}

{\small

\begin{center}
\begin{tabular}{|rr|r|rr|r|rr|r|rr|}
\multicolumn{11}{c} {$n=11$,\  $\lambda_{1} \ge 5$} \\
\hline
$\lambda$ &  $\eta_{\lambda}$ & \ \ \ & $\lambda$ &  $\eta_{\lambda}$ & \ \ \  &  $\lambda$ & $\eta_{\lambda}$&  \ \ \ &$\lambda$ & $\eta_{\lambda}$ \\
\hline
$  11 $ & $ 14684570  $ & & $  7, 3, 1  $ & $ 3338   $  & & $ 6, 2^{2}, 1$ & $-557 $ & & $ 5, 3, 1^{3}$  & $134$ \\
$  10, 1 $ & $ -1468457  $ & & $ 7, 2^{2}  $ & $  3178   $  & & $6, 2, 1^{3} $ & $-530 $ & & $5, 2^{3} $  & $122$ \\
$  9, 2 $ & $ 166870  $ & & $ 7, 2, 1^{2}  $ & $ 3090   $  & & $6, 1^{5} $ & $-485 $ & & $5, 2^{2}, 1^{2} $  & $118$ \\
$ 9, 1^{2}  $ & $ 163162  $ & & $ 7, 1^{4}   $ & $  2914  $  & & $5^{2}, 1 $ & $362 $ & & $5, 2, 1^{4} $  & $ 110$ \\
$ 8, 3  $ & $  -22249  $ & & $ 6, 5  $ & $ -905   $  & & $5, 4, 2 $ & $ 178 $ & & $5, 1^{6} $  & $98$ \\
$ 8, 2, 1  $ & $  -21190 $ & & $ 6, 4, 1  $ & $ -710   $  & & $ 5, 4, 1^{2}$ & $170 $ & & $ $  & $$ \\
$ 8, 1^{3}  $ & $ -20395  $ & & $ 6, 3, 2  $ & $ -617   $  & & $ 5, 3^{2} $ & $158 $ & & $ $  & $$ \\
$ 7, 4  $ & $  3706  $ & & $ 6, 3, 1^{2}  $ & $ -595   $  & & $5, 3, 2, 1 $ & $143 $ & & $ $  & $$  \\

\hline
\end{tabular}
\end{center}

\begin{center}
\begin{tabular}{|rr|r|rr|r|rr|r|rr|}
\multicolumn{11}{c} {$n=12$,\  $\lambda_{1} \ge 6$} \\
\hline
$\lambda$ &  $\eta_{\lambda}$ & \ \ \ & $\lambda$ &  $\eta_{\lambda}$ & \ \ \  &  $\lambda$ & $\eta_{\lambda}$&  \ \ \ &$\lambda$ & $\eta_{\lambda}$ \\
\hline
$ 12  $ & $ 176214841  $ & & $ 8, 3, 1   $ & $ 24721   $  & & $ 7, 2^{2}, 1 $ & $ -3531$ & & $6, 3, 2, 1 $  & $694$ \\
$   11, 1 $ & $  -16019531  $ & & $ 8, 2^{2}   $ & $  23839   $  & & $7, 2, 1^{3}  $ & $ -3399$ & & $6, 3, 1^{3}  $  & $661$ \\
$ 10, 2  $ & $  1631619  $ & & $8, 2, 1^{2}  $ & $23309  $  & & $7, 1^{5} $ & $  -3179$ & & $ 6, 2^{3}$  & $637$ \\
$ 10, 1^{2}  $ & $ 1601953  $ & & $ 8, 1^{4}   $ & $ 22249   $  & & $ 6^{2} $ & $ 2119$ & & $6, 2^{2}, 1^{2} $  & $619 $ \\
$ 9, 3  $ & $ -190709   $ & & $7, 5 $ & $ -4959   $  & & $ 6, 5, 1$ & $ 1033 $ & & $6, 2, 1^{4}  $  & $583$ \\
$ 9, 2, 1  $ & $ -183557  $ & & $ 7, 4, 1  $ & $  -4169   $  & & $6, 4, 2 $ & $ 829 $ & & $6, 1^{6} $  & $ 529$ \\
$ 9, 1^{3}  $ & $ -177995  $ & & $7, 3, 2  $ & $  -3815  $  & & $6, 4, 1^{2}  $ & $ 799$ & & $ $  & $$ \\
$ 8, 4  $ & $  26701   $ & & $ 7, 3, 1^{2}  $ & $ -3709   $  & & $6, 3^{2} $ & $ 739$ & & $ $  & $$  \\

\hline
\end{tabular}
\end{center}

\begin{center}
\begin{tabular}{|rr|r|rr|r|rr|r|rr|}
\multicolumn{11}{c} {$n=13$, \ $\lambda_{1} \ge 6$} \\
\hline
$\lambda$ &  $\eta_{\lambda}$ & \ \ \ & $\lambda$ &  $\eta_{\lambda}$ & \ \ \  &  $\lambda$ & $\eta_{\lambda}$&  \ \ \ &$\lambda$ & $\eta_{\lambda}$ \\
\hline
$ 13  $ & $  2290792932  $ & & $  9, 1^{4}   $ & $ 192828   $  & & $7, 4, 1^{2} $ & $4632 $ & & $6, 4, 3 $  & $-996$ \\
$ 12, 1  $ & $ -190899411  $ & & $ 8, 5  $ & $ -33363    $  & &$7,3^{2} $ & $4452$ & & $6, 4, 2, 1 $  & $ -933$ \\
$ 11, 2  $ & $ 17621484  $ & & $ 8, 4, 1 $ & $ -29668  $  & & $7, 3, 2, 1 $ & $4239 $ & & $ 6, 4, 1^{3} $  & $-888$ \\
$ 11, 1^{2}  $ & $  17354492  $ & & $ 8, 3, 2   $ & $  -27811   $  & & $7, 3, 1^{3}  $ & $ 4080$ & & $6, 3^{2}, 1 $  & $ -831$ \\
$ 10, 3  $ & $ -1835571   $ & & $8, 3, 1^{2} $ & $   -27193  $  & & $7, 2^{3} $ & $ 3972  $ & & $6, 3, 2^{2} $  & $ -793$ \\
$ 10, 2, 1  $ & $ -1779948  $ & & $ 8, 2^{2}, 1  $ & $  -26223   $  & & $ 7, 2^{2}, 1^{2} $ & $3884 $ & & $6, 3, 2, 1^{2} $  & $-771$ \\
$ 10, 1^{3}  $ & $ -1735449  $ & & $ 8, 2, 1^{3} $ & $  -25428   $  & & $ 7, 2, 1^{4} $ & $3708 $ & & $6, 3, 1^{4} $  & $-727$ \\
$ 9, 4  $ & $  222492   $ & & $ 8, 1^{5}  $ & $  -24103   $  & & $7, 1^{6} $ & $3444 $ & & $ 6, 2^{3}, 1$  & $-708$  \\
$ 9, 3, 1  $ & $ 209780  $ & & $7, 6   $ & $ 7284   $  & & $6^{2}, 1  $ & $-2428 $ & & $6, 2^{2}, 1^{3} $  & $-681$ \\
$ 9, 2^{2}  $ & $  203952  $ & & $7, 5, 1  $ & $ 5580   $  & & $6, 5, 2  $ & $ -1203 $ & & $ 6, 2, 1^{5}$  & $ -636$ \\
$  9, 2, 1^{2} $ & $ 200244   $ & & $7, 4, 2   $ & $  4764  $  & & $6, 5, 1^{2} $ & $-1161 $ & & $6, 1^{7} $  & $-573$  \\

\hline
\end{tabular}
\end{center}
}

{\small

\begin{minipage}[t]{5cm}
\begin{tabular}{|rr|r|rr|r|rr|r|rr|}
\multicolumn{11}{c} {$n=15$} \\
\hline
$\lambda$ &  $\eta_{\lambda}$ & \ \ \ & $\lambda$ &  $\eta_{\lambda}$ & \ \ \  &  $\lambda$ & $\eta_{\lambda}$&  \ \ \ &$\lambda$ & $\eta_{\lambda}$ \\
\hline
$ 15  $ & $  481066515734 $ & & $  7^{2}, 1   $ & $ 18806  $ & & $ 6, 2, 1^{7} $ & $-742$ & &$4, 3^{2}, 2^{2}, 1$ & $-77$ \\
$ 14, 1 $ & $ -34361893981 $ & & $  7, 6, 2  $ & $ 9350 $  & & $ 6, 1^{9} $ & $-661$ & & $4, 3^{2}, 2, 1^{3}$ & $-74$ \\
$ 13, 2  $ & $  2672591754 $ & & $  7, 6, 1^{2} $ & $  9094 $  & & $ 5^{3} $ & $1214$ & & $4, 3^{2}, 1^{5}$ & $ -69$   \\
$ 13, 1^{2} $ & $ 2643222614 $ & & $ 7, 5, 3   $ & $ 7446 $ & & $5^{2}, 4, 1 $ & $859$ & & $4, 3, 2^{4}$ & $-73$ \\
$ 12, 3  $ & $ -229079293 $ & & $ 7, 5, 2, 1  $ & $  7089 $ & & $5^{2}, 3, 2 $ & $742$ & & $4, 3, 2^{3}, 1^{2}$ & $-71$  \\
$ 12, 2, 1  $ & $ -224273434 $ & & $ 7, 5, 1^{3}  $ & $   6822 $ & &$5^{2}, 3, 1^{2} $ & $714$ & & $ 4, 3, 2^{3}, 1^{4}$ & $-67$  \\
$ 12, 1^{3}  $ & $  -220268551 $ & & $  7, 4^{2} $ & $  6662 $ & & $5^{2}, 2^{2}, 1 $ & $662$& & $4, 3, 2, 1^{6}$ & $-61$  \\
$ 11, 4  $ & $  22026854  $ & &$ 7, 4, 3, 1  $ & $  6174 $ & &$5^{2},2,1^{3}  $ & $629$ & & $ 4, 3, 1^{8}$ & $-53$   \\
$ 11, 3, 1 $ & $ 21211046  $ & & $ 7, 4, 2^{2} $ & $  5954 $ & &$5^{2}, 1^{5}$ & $574$ & & $4, 2^{5}, 1 $ & $-66$  \\
$ 11, 2^{2} $ & $  20825390 $ & & $ 7, 4, 2, 1^{2} $ & $ 5822 $ & & $5, 4^{2}, 2$ & $374$& & $4, 2^{4}, 1^{3}$ & $-63$  \\
$ 11, 2, 1^{2}  $ & $  20558398 $ & & $ 7, 4, 1^{4} $ & $ 5558 $ & &$5, 4^{2}, 1^{2}$ & $362 $ & & $ 4, 2^{3}, 1^{5}$ & $-58$ \\
$ 11, 1^{4}  $ & $ 20024414 $ & & $ 7, 3^{2}, 2 $ & $ 5566 $ & &$5, 4, 3^{2}$ & $350$ & & $4, 2^{2}, 1^{7}$ & $-51$ \\
$ 10, 5   $ & $  -2447421 $ & & $ 7, 3^{2}, 1^{2}$ & $ 5442 $ & &$5, 4, 3, 2, 1$ & $329$& & $4, 2, 1^{9}$ & $-42$ \\
$ 10, 4, 1 $ & $ -2288506 $ & &$ 7, 3, 2^{2}, 1 $ & $ 5246 $ & & $5, 4, 3, 1^{3}$ & $314$& & $4, 1^{11} $ & $-31$ \\
$ 10, 3, 2 $ & $ -2202685 $ & & $ 7, 3, 2, 1^{3} $ & $ 5087 $ & & $5, 4, 2^{3} $ & $302 $& & $ 3^{5} $ & $134$ \\
$ 10, 3, 1^{2}  $ & $   -2169311 $ & & $ 7, 3, 1^{5} $ & $ 4822 $ & &$5, 4, 2^{2}, 1^{2}$ & $294$  & & $3^{4}, 2, 1$  &  $119$  \\
$ 10, 2^{2}, 1  $ & $  -2121105 $ & & $ 7, 2^{4} $ & $ 4854 $ & & $5, 4, 2, 1^{4}$ &  $278$ & & $3^{4}, 1^{3}$ & $110$ \\
$ 10, 2, 1^{3}   $ & $  -2076606 $ & & $ 7, 2^{3}, 1^{2} $ & $ 4766 $ & &$5, 4, 1^{6} $ & $254$ & & $3^{3}, 2^{3}$ & $98$  \\
$ 10, 1^{5} $ & $ -2002441 $ & & $ 7, 2^{2}, 1^{4} $ & $ 4590 $  & &$5, 3^{3}, 1 $ & $290 $ & & $3^{3}, 2^{2}, 1^{2}$ & $94$ \\
$  9, 6  $ & $ 333674 $ & & $ 7, 2, 1^{6} $ & $ 4326 $ & & $5, 3^{2}, 2^{2}$ & $274$& &  $3^{3}, 2, 1^{4}$ & $86$\\
$ 9, 5, 1   $ & $ 293702 $ & & $ 7, 1^{8} $ & $ 3974 $ & & $5, 3^{2}, 2, 1^{2}$ & $266$& &  $3^{3}, 1^{6}$ & $74$\\
$ 9, 4, 2  $ & $ 271934 $ & & $ 6^{2}, 3 $ & $ -3430 $ & & $5, 3^{2}, 1^{4}$ & $250$& &  $3^{2}, 2^{4}, 1$ & $62$ \\
$ 9, 4, 1^{2}  $ & $  266990 $ & & $ 6^{2}, 2, 1 $ & $-3205$ & & $5, 3, 2^{3}, 1$ & $239$& &  $3^{2}, 2^{3}, 1^{3}$ & $59$ \\
$ 9, 3^{2}  $ & $  262226 $ & & $ 6^{2}, 1^{3} $ &  $-3046$  & &$5, 3, 2^{2}, 1^{3}$ & $230$ & &  $3^{2}, 2^{2}, 1^{5}$ & $54$\\
$ 9, 3, 2, 1  $ & $ 254279  $ & & $ 6, 5, 4 $ &  $-1789$  & & $5, 3, 2, 1^{5}$ & $215$& & $3^{2}, 2, 1^{7}$ & $47$ \\
$  9, 3, 1^{3} $ & $  247922 $ & & $ 6, 5, 3, 1 $ & $-1617$  & &$5, 3, 1^{7}$ & $194$ & & $3^{2}, 1^{9}$ & $38$ \\
$   9, 2^{3} $ & $ 244742 $ & & $6, 5, 2^{2} $ &   $-1543$  & &$5, 2^{5}$ & $194$  & & $3, 2^{6}$ & $14$ \\
$   9, 2^{2}, 1^{2} $ & $ 241034 $ & & $6, 5, 2, 1^{2}$ & $ -1501$ & &$5, 2^{4}, 1^{2}$ & $190$ & & $3, 2^{5}, 1^{2} $ & $14$  \\
$  9, 2, 1^{4} $ & $ 233618 $ & & $ 6, 5, 1^{4} $ & $-1417$  & &$5, 2^{3}, 1^{4}$ & $182$& & $3, 2^{4}, 1^{4} $ & $14$ \\
$  9, 1^{6}  $ & $ 222494 $ & & $6, 4^{2}, 1 $ & $-1411$  & &$5, 2^{2}, 1^{6} $ & $170$ & & $3, 2^{3}, 1^{6}$ & $14$ \\
$ 8, 7  $ & $ -65821 $ & & $6, 4, 3, 2 $ & $-1282$  & &$5, 2, 1^{8}$ & $154$& & $3, 2^{2}, 1^{8}$ & $14$ \\
$ 8, 6, 1  $ & $ -49546 $ & &$6, 4, 3, 1^{2} $ & $ -1246$   & &$5, 1^{10} $ & $134$ & & $3, 2, 1^{10}$ & $14$\\
$ 8,5,2 $    &   $-41701$     & & $6,4,2^{2},1$ &  $-1181$      & & $4^{3}, 3$         & $-331$   & & $3, 1^{12}$ & $14$ \\
$ 8, 5, 1^{2}  $ & $  -40775 $ & & $6, 4, 2, 1^{3} $ & $-1141$  & &$4^{3}, 2, 1$ & $-298$ & & $2^{7}, 1$ & $-49$\\
$  8, 4, 3  $ & $ -38146 $ & & $6, 4, 1^{5} $ & $-1066$  & &$4^{3}, 1, 1, 1 $ & $-277$ & & $2^{6}, 1^{3}$ & $-46$ \\
$ 8, 4, 2, 1  $ & $ -36715 $ & & $6, 3^{3} $ & $-1105$  & &$4^{2}, 3^{2}, 1 $ & $-226$ & & $2^{5}, 1^{5}$ & $-41$\\
$ 8, 4, 1^{3}  $ & $  -35602 $ & & $6, 3^{2}, 2, 1$ & $-1054$  & &$4^{2}, 3, 2^{2}$ & $-210$ & & $2^{4}, 1^{7} $ & $-34$  \\
$  8, 3^{2}, 1 $ & $ -34961 $ & & $6, 3^{2}, 1^{3}$ & $-1015$   & & $4^{2}, 3, 2, 1^{2}$ & $-202$& & $2^{3}, 1^{9} $ & $-25$\\
$ 8, 3, 2^{2}   $ & $ -33991 $ & & $6, 3, 2^{3}$ & $ -991$   & & $4^{2}, 3, 1^{4}$ & $-186$& & $2^{2}, 1^{11} $ & $-14$\\
$  8, 3, 2, 1^{2}  $ & $  -33373 $ & & $6, 3, 2^{2}, 1^{2} $ & $-969$   & &$4^{2}, 2^{3}, 1 $ & $-175$ & & $2, 1^{13} $ & $-1$ \\
$   8, 3, 1^{4} $ & $ -32137 $ & & $6, 3, 2, 1^{4}$ & $ -925$    &&$4^{2}, 2^{2}, 1^{3} $ & $-166$ & & $1^{15}$ & $14$\\
$ 8, 2^{3}, 1  $ & $ -31786 $ & & $ 6, 3, 1^{6} $ & $-859$    & &$4^{2}, 2, 1^{5}$ & $-151$ & & $ $   & $ $\\
$ 8, 2^{2}, 1^{3}  $ & $  -30991 $ & & $ 6, 2^{4}, 1 $ & $ -877$  & & $4^{2}, 1^{7}$ & $-130$& & &\\
$  8, 2, 1^{5} $ & $ -29666 $ & & $ 6, 2^{3}, 1^{3} $ & $-850$  & & $4, 3^{3}, 2$ & $-81$& &  &\\
$  8, 1^{7}  $ & $  -27811 $ & & $6, 2^{2}, 1^{5} $ & $ -805$  & & $4, 3^{3}, 1^{2}$ & $-79$& &  &\\
\hline
\end{tabular}
\end{minipage}

}

\end{section}


\begin{thebibliography}{99}

\bibitem{Babai} L. Babai, Spectra of Cayley graphs, {\em J. Combin. Theory Ser. B} {\bf 27} (1979), 180--189.

\bibitem{DS} P. Diaconis and M. Shahshahani, Generating a random permutation with random transpositions, {\em Zeit. f{\" u}r Wahrscheinlichkeitstheorie} {\bf 57} (1981), 159--179.

\bibitem{EFP} D. Ellis, E. Friedgut and H. Pilpel, Intersecting families of permutations, {\em Journal of the American Society} {\bf 24} (2011), 649-682.

\bibitem{Ku-Wales} C. Y. Ku and D. B. Wales, Eigenvalues of the derangement graph, {\em J. Combin. Theory Ser. A} {\bf 117} (2010), 289--312.

\bibitem{Lub} A. Lubotzky, Discrete Groups, Expanding Graphs, and Invariant Measures, {\it Progress in
Mathematics}, vol. {\bf 125}, Birkh\"auser, 1994.

\bibitem{Ram} M. Ram Murty, Ramanujan graphs, {\em J. Ramanujan Math. Soc.} {\bf 18} (2003), 1--20.


\bibitem{OO} A. Okounkov and G. Olshanski, Shifted schur functions, {\em Algebra i Analiz} {\bf 9}:2 (1997), 73--146.

\bibitem{Renteln} P. Renteln, On the spectrum of the derangement graph, {\em The Electronic Journal of Combinatorics} {\bf 14} (2007), \# R82.

\end{thebibliography}
\end{document}